\documentclass[english]{amsart}
\usepackage[T1]{fontenc}
\usepackage[latin9]{inputenc}
\usepackage{amstext}
\usepackage{amsthm}
\usepackage{amssymb}

\makeatletter
\numberwithin{equation}{section}
\numberwithin{figure}{section}
\theoremstyle{plain}
\newtheorem{thm}{\protect\theoremname}
\theoremstyle{plain}
\newtheorem{cor}[thm]{\protect\corollaryname}
\theoremstyle{remark}
\newtheorem{rem}[thm]{\protect\remarkname}

\makeatother

\usepackage{babel}
\providecommand{\corollaryname}{Corollary}
\providecommand{\remarkname}{Remark}
\providecommand{\theoremname}{Theorem}

\begin{document}

\title{Summing Lambert Series in Euler's q-Exponential Functions}

\author{Ruiming Zhang}

\curraddr{College of Science\\
Northwest A \& F University\\
Yangling, Shaanxi 712100\\
P. R. China.}

\email{ruimingzhang@yahoo.com}

\begin{abstract}
In the work we shall present formulas to sum Lambert series using
Euler's q-exponential functions, and several Lambert series associated
with well-known arithmetic functions are given as examples. These
functions are: the M\"{o}bius $\mu(n)$, the Euler's totient $\varphi(n)$,
Jordan's totient $J_{k}(n)$, von Mangoldt $\Lambda(n)$, divisor
function $\sigma_{s}(n)$, the Ramanujan's sum $c_{q}(n)$ , and sum
of square functions $r_{2}(n),r_{4}(n),r_{8}(n)$.
\end{abstract}

\thanks{The work is supported by the National Natural Science Foundation
of China grants No. 11371294 and No. 11771355. }

\maketitle

\section{Introduction}

Lambert series
\[
\sum_{n=1}^{\infty}\frac{a_{n}q^{n}}{1-q^{n}}
\]
is a very important class of $q$-series that often appears in the
studies of arithmetic functions, summability methods and modular forms.
Many of them are also generating functions for well-known arithmetic
functions $a_{n}$, for example,
\[
\sum_{n=1}^{\infty}\frac{\mu(n)q^{n}}{1-q^{n}}=q,\quad\sum_{n=1}^{\infty}\frac{\varphi(n)q^{n}}{1-q^{n}}=\frac{q}{(1-q)^{2}},
\]
where $\mu(n)$ and $\varphi(n)$ are the M\"{o}bius function and
Euler totient function respectively. In this work we shall evaluate
several Lambert series expansions associated various arithmetic functions
in terms of Euler q-exponential functions, their forms are reminiscent
of the formulas \cite{Apostol,Sandor2}

\[
\prod_{n=1}^{\infty}\left(1-q^{n}\right)^{\varphi(n)/n}=\exp\left(-\frac{q}{1-q}\right)
\]
and 
\[
\prod_{n=0}^{\infty}\left(\frac{1+q^{2n+1}}{1-q^{2n+1}}\right)^{\varphi(2n+1)/(2n+1)}=\exp\left(\frac{2q}{1-q^{2}}\right).
\]
In the following we shall use the following common notations \cite{Andrews,DLMF,Gasper,Ismail}
\begin{equation}
\left(z;q\right)_{\infty}=\prod_{n=0}^{\infty}\left(1-zq^{n}\right),\quad\left(z_{1},\dots z_{m};q\right)_{\infty}=\prod_{k=1}^{m}\left(z_{k};q\right)_{\infty},\label{eq:1.1}
\end{equation}
where $\left|q\right|<1$, $m\in\mathbb{N}$ and $z,z_{1},\dots,z_{m}\in\mathbb{C}$.
The q-shifted factorials can be defined by
\begin{equation}
\left(z;q\right)_{n}=\frac{\left(z;q\right)_{\infty}}{\left(zq^{n};q\right)_{\infty}},\ \left(z_{1},\dots z_{m};q\right)_{n}=\prod_{k=1}^{m}\left(z_{k};q\right)_{n},\quad n\in\mathbb{C}.\label{eq:1.2}
\end{equation}
From the q-binomial theorem, \cite{Andrews,DLMF,Gasper,Ismail}
\begin{equation}
\frac{(az;q)_{\infty}}{(z;q)_{\infty}}=\sum_{n=0}^{\infty}\frac{(a;q)_{n}}{(q;q)_{n}}z^{n},\quad|z|<1,\label{eq:1.3}
\end{equation}
two special identities can be derived quickly, the first one is to
let $a=0$ in (\ref{eq:1.3}),
\[
e_{q}(z)=\frac{1}{(z;q)_{\infty}}=\sum_{n=0}^{\infty}\frac{z^{n}}{(q;q)_{n}},\quad|z|<1,
\]
while the second is obtained by replacing $z$ by $-z/a$ (\ref{eq:1.3}),
then letting $a\to\infty$,
\[
E_{q}(z)=(-z;q)_{\infty}=\sum_{n=0}^{\infty}\frac{q^{\binom{n}{2}}}{(q;q)_{n}}z^{n},\quad z\in\mathbb{C}.
\]
The two functions $e_{q}(z)$ and $E_{q}(z)$ are called Euler's q-exponential
functions, they are two among infinitely many q-analogues of the classical
exponential function $e^{z}$. 

Many fundamental functions in the theory of special functions can
be defined using Euler q-exponentials. For example, the q-Gamma function,
\cite{Andrews,DLMF,Gasper,Ismail}
\[
\frac{1}{\Gamma_{q}(z)}=\frac{\left(q^{z};q\right)_{\infty}}{\left(q;q\right)_{\infty}}\left(1-q\right)^{z-1},\quad z\in\mathbb{C}.
\]
 and the basic hypergeometric series $_{r}\phi_{s}$, 
\[
\begin{aligned}_{r}\phi_{s}\left(\begin{array}{cc}
\begin{array}{c}
a_{1},\dots,a_{r}\\
b_{1},\dots b_{s}
\end{array} & \bigg|q,z\end{array}\right) & =\sum_{n=-\infty}^{\infty}\frac{\left(a_{1},\dots,a_{r};q\right)_{n}z^{n}}{\left(q,b_{1},\dots,b_{s};q\right)_{n}}\left(-q^{\frac{n-1}{2}}\right)^{n\left(s+1-r\right)}.\end{aligned}
\]
It is not an exaggeration to say that Euler's q-exponential functions
are at the heart of q-special functions. It is worth noticing that
our definitions for the $q$-shifted factorial $(z;q)_{n}$ in (\ref{eq:1.2})
is based on the definitions of q-exponentials, thus the index in $(z;q)_{n}$
can be any complex numbers. This promotes the view that almost all
the the summations in q-series are actually over the full group $\mathbb{Z}$,
which is a nice analytical feature of many q-special functions.

Since all the $\theta$-functions are related to each other, for our
convenience we examine a simpler one, \cite{Andrews,DLMF,Gasper,Ismail}
\[
\sum_{n=-\infty}^{\infty}q^{n^{2}/2}\left(-z\right)^{n}=\left(q,q^{1/2}z,q^{1/2}/z;q\right)_{\infty},\quad|q|<1,\ z\neq0.
\]
Clearly, it is nothing but a product of three q-exponentials. Since
the identity,
\[
\frac{\sum_{n=-\infty}^{\infty}(-1)^{n}q^{n^{2}/2+nw}}{(1-q)(q;q)_{\infty}^{2}}=\frac{1}{\Gamma_{q}\left(\frac{1}{2}-w\right)\Gamma_{q}\left(\frac{1}{2}+w\right)},\quad w\in\mathbb{C}
\]
is an obvious analogue to the reflection formula for the Euler $\Gamma$
function $\Gamma(w)$. Consequently, the $\theta$-functions are just
the q-analogues of trigonometric functions. The Euler q-exponential
functions are also closely related to the Dedekind $\eta(\tau)$ function
through 
\begin{equation}
\eta(\tau)=q^{1/24}\prod_{n=1}^{\infty}\left(1-q^{n}\right)=q^{1/24}\left(q;q\right)_{\infty},\label{eq:1.4}
\end{equation}
where $q=e^{2\pi i\tau},\ \Im(\tau)>0$. Since $\eta(\tau)$ can be
used to construct other modular forms, for example, the modular discriminant
of Weierstrass, 
\begin{equation}
\Delta(\tau)=\left(2\pi\right)^{12}q\left(q;q\right)_{\infty}^{24},\label{eq:1.5}
\end{equation}
thus it is far-fetched to speculate that the Euler q-exponential functions
is also behind some important properties of modular forms. 

\section{Main Results}
\begin{thm}
\label{thm:2.1}Let $f(n),g(n)$ be an arithmetic functions such that
\begin{equation}
f(n)=\sum_{d\vert n}g(d),\quad\sum_{n=1}^{\infty}\frac{\left|g(n)\right|}{n}c^{n}<\infty\label{eq:2.1}
\end{equation}
for all $0<c<1$. If $\Re(z)>0,\,0<q<1$, then, 
\begin{equation}
\prod_{n=1}^{\infty}\left(q^{nz};q^{n}\right)_{\infty}^{g(n)/n}=\exp\left(-\sum_{n=1}^{\infty}\frac{f(n)}{n}\frac{q^{nz}}{1-q^{n}}\right).\label{eq:2.2}
\end{equation}
 and
\begin{equation}
\prod_{n=1}^{\infty}\left(\frac{\left(q^{n(z+1)};q^{2n}\right)_{\infty}}{\left(q^{nz};q^{2n}\right)_{\infty}}\right)^{g(n)/n}=\exp\left(\sum_{n=1}^{\infty}\frac{f(n)}{n}\frac{q^{nz}}{1+q^{n}}\right).\label{eq:2.3}
\end{equation}
\end{thm}

\begin{proof}
Formula (\ref{eq:2.2}) is equivalent to 
\begin{equation}
-\sum_{n=1}^{\infty}\frac{g(n)}{n}\log\left(q^{nz};q^{n}\right)_{\infty}=\sum_{n=1}^{\infty}\frac{f(n)}{n}\frac{q^{nz}}{1-q^{n}}.\label{eq:2.4}
\end{equation}
 Since $0<q<1$ and $\Re(z)>0$, and 
\[
\sum_{n=1}^{\infty}\frac{\left|g(n)\right|}{n}c^{n}<\infty,\quad0<c<1
\]
 imply that
\[
0<-\sum_{n=1}^{\infty}\frac{\left|g(n)\right|}{n}\sum_{j=0}^{\infty}\log\left(1-c^{n\Re(z)+jn}\right)<\infty.
\]
Then by
\[
-\log(1-x)=\sum_{n=1}^{\infty}\frac{x^{n}}{n},\quad0<x<1,
\]
and Fubini's theorem we have
\begin{align*}
 & -\sum_{n=1}^{\infty}\frac{g(n)}{n}\log\left(q^{nz};q^{n}\right)_{\infty}=-\sum_{n=1}^{\infty}\frac{g(n)}{n}\sum_{j=0}^{\infty}\log\left(1-q^{nz+jn}\right)\\
 & =\sum_{n=1}^{\infty}\frac{g(n)}{n}\sum_{j=0}^{\infty}\sum_{k=1}^{\infty}\frac{q^{k(nz+jn)}}{k}=\sum_{j=0}^{\infty}\sum_{n=1}^{\infty}\sum_{k=1}^{\infty}\frac{g(n)}{kn}q^{nk(z+j)}=\sum_{j=0}^{\infty}\sum_{m=1}^{\infty}\left(\sum_{d\vert m}g(m)\right)\frac{q^{m(z+j)}}{m}\\
 & =\sum_{j=0}^{\infty}\sum_{m=1}^{\infty}\frac{f(m)}{m}q^{mz+mj}=\sum_{m=1}^{\infty}\frac{f(m)}{m}\frac{q^{mz}}{1-q^{m}},
\end{align*}
which is (\ref{eq:2.4}). Since
\[
\left(q^{nz};q^{n}\right)_{\infty}=\left(q^{nz};q^{2n}\right)_{\infty}\left(q^{n(z+1)};q^{2n}\right)_{\infty},
\]
then on one hand we have
\begin{align*}
 & -\sum_{n=1}^{\infty}\frac{2g(n)}{n}\log\left(q^{nz};q^{2n}\right)_{\infty}+\sum_{n=1}^{\infty}\frac{g(n)}{n}\log\left(q^{nz};q^{n}\right)_{\infty}\\
 & =\sum_{n=1}^{\infty}\frac{g(n)}{n}\left(\log\left(\frac{\left(q^{nz};q^{2n}\right)_{\infty}\left(q^{n(z+1)};q^{2n}\right)_{\infty}}{\left(q^{nz};q^{2n}\right)_{\infty}^{2}}\right)\right)=\sum_{n=1}^{\infty}\frac{g(n)}{n}\left(\log\left(\frac{\left(q^{n(z+1)};q^{2n}\right)_{\infty}}{\left(q^{nz};q^{2n}\right)_{\infty}}\right)\right).
\end{align*}
On the other hand, by (\ref{eq:2.4}) we obtain
\begin{align*}
 & -\sum_{n=1}^{\infty}\frac{2g(n)}{n}\log\left(q^{nz};q^{2n}\right)_{\infty}+\sum_{n=1}^{\infty}\frac{g(n)}{n}\log\left(q^{nz};q^{n}\right)_{\infty}\\
 & =\sum_{n=1}^{\infty}\frac{f(m)}{m}\left(\frac{2q^{mz}}{1-q^{2m}}-\frac{q^{mz}}{1-q^{m}}\right)=\sum_{n=1}^{\infty}\frac{f(m)}{m}\frac{q^{mz}}{1+q^{m}},
\end{align*}
hence,
\[
\sum_{n=1}^{\infty}\frac{g(n)}{n}\left(\log\left(\frac{\left(q^{n(z+1)};q^{2n}\right)_{\infty}}{\left(q^{nz};q^{2n}\right)_{\infty}}\right)\right)=\sum_{n=1}^{\infty}\frac{f(m)}{m}\frac{q^{mz}}{1+q^{m}},
\]
which is equivalent to (\ref{eq:2.3}).
\end{proof}
\begin{cor}
\label{cor:2.2} If $f(n),g(n),q,z$ as defined in Theorem \ref{thm:2.1},
let us define 
\begin{equation}
h(n)=\sum_{d\vert n}\frac{\mu(d)}{d}f\left(\frac{n}{d}\right),\label{eq:2.5}
\end{equation}
 where $\mu(d)$ is the M\"{o}bius function. Then,
\begin{equation}
nh(n)=\sum_{d\vert n}dg(d)\varphi\left(\frac{n}{d}\right)=\sum_{k=1}^{n}\gcd(n,k)g\left(\gcd(n,k)\right).\label{eq:2.6}
\end{equation}
 Furthermore, we have
\begin{equation}
\prod_{n=1}^{\infty}\left(q^{nz};q^{n}\right)_{\infty}^{h(n)}=\exp\left(-\sum_{m=1}^{\infty}\frac{f(m)q^{mz}}{1-q^{m}}\right)\label{eq:2.7}
\end{equation}
and
\begin{equation}
\prod_{n=1}^{\infty}\left(\frac{\left(q^{n(z+1)};q^{2n}\right)_{\infty}}{\left(q^{nz};q^{2n}\right)_{\infty}}\right)^{h(n)}=\exp\left(\sum_{m=1}^{\infty}\frac{f(m)q^{mz}}{1+q^{m}}\right).\label{eq:2.8}
\end{equation}
 
\end{cor}

\begin{proof}
By (\ref{eq:2.5}) we have
\[
nh(n)=\sum_{d\vert n}\mu(d)\frac{n}{d}f\left(\frac{n}{d}\right)=\sum_{d\vert n}df(d)\mu\left(\frac{n}{d}\right),
\]
hence,
\[
nf(n)=\sum_{d\vert n}dh(d).
\]
Then (\ref{eq:2.7}) and (\ref{eq:2.8}) follow from (\ref{eq:2.2})
and (\ref{eq:2.3}) respectively. 

Let 
\[
1(n)=1,\quad\mu^{(1)}(n)=\frac{\mu(n)}{n},\quad\varphi^{(1)}(n)=\frac{\varphi(n)}{n},
\]
where $\varphi(n)$ is the Euler's totient function. We rewrite (\ref{eq:2.1})
and (\ref{eq:2.5}) as a Dirichlet convolution,
\[
f=1*g,\quad h=\mu^{(1)}*f,
\]
then by 
\[
\frac{\varphi(n)}{n}=\sum_{d\vert n}\frac{\mu(d)}{d},
\]
 we get
\[
h=\mu^{(1)}*f=\mu^{(1)}*\left(1*g\right)=\left(\mu^{(1)}*1\right)*g=\varphi^{(1)}*g,
\]
which is 
\[
h(n)=\sum_{d\vert n}\frac{\varphi(d)}{d}g\left(\frac{n}{d}\right).
\]
Hence
\[
nh(n)=\sum_{d\vert n}\varphi(d)\frac{n}{d}g\left(\frac{n}{d}\right)=\sum_{d\vert n}dg(d)\varphi\left(\frac{n}{d}\right).
\]
For any arithmetic function $f(n)$, since \cite{Gould}
\[
\sum_{n=1}^{\infty}\frac{1}{n^{s}}\sum_{k=1}^{n}f\left(\gcd(n,k)\right)=\frac{\zeta(s-1)}{\zeta(s)}\sum_{n=1}^{\infty}\frac{f(n)}{n^{s}},
\]
 then,
\[
\sum_{d\vert n}f(d)\varphi\left(\frac{n}{d}\right)=\sum_{k=1}^{n}f\left(\gcd(n,k)\right).
\]
 Hence,
\[
h(n)=\frac{1}{n}\sum_{k=1}^{n}\gcd(n,k)g\left(\gcd(n,k)\right).
\]
\end{proof}
\begin{rem}
If $\Re(z)>0$ and $0<q<1$, by Corollary \ref{cor:2.2}, then
\begin{equation}
\prod_{n=1}^{\infty}\left(q^{nz};q^{n}\right)_{\infty}^{\sum_{k=1}^{n}\gcd(n,k)g\left(\gcd(n,k)\right)/n}=\exp\left(-\sum_{m=1}^{\infty}\frac{\left(\sum_{d\vert m}g(d)\right)q^{mz}}{1-q^{m}}\right)\label{eq:2.9}
\end{equation}
and
\begin{equation}
\prod_{n=1}^{\infty}\left(\frac{\left(q^{n(z+1)};q^{2n}\right)_{\infty}}{\left(q^{nz};q^{2n}\right)_{\infty}}\right)^{\sum_{k=1}^{n}\gcd(n,k)g\left(\gcd(n,k)\right)/n}=\exp\left(\sum_{m=1}^{\infty}\frac{\left(\sum_{d\vert m}g(d)\right)q^{mz}}{1+q^{m}}\right),\label{eq:2.10}
\end{equation}
where $g(n)$ is any arithmetic function.
\end{rem}

\begin{rem}
The identity
\begin{equation}
\sum_{d\vert n}df(d)\mu\left(\frac{n}{d}\right)=\sum_{d\vert n}dg(d)\varphi\left(\frac{n}{d}\right)=\sum_{k=1}^{n}\gcd(n,k)g\left(\gcd(n,k)\right)\label{eq:2.11}
\end{equation}
 is very curious. Let $f(n)=n^{\alpha},\alpha\in\mathbb{C}$, then
\[
g(n)=n^{\alpha}\prod_{p\vert n}\left(1-\frac{1}{p^{\alpha}}\right)=J_{\alpha}(n),\quad\sum_{d\vert n}df(d)\mu\left(\frac{n}{d}\right)=J_{\alpha+1}(n),
\]
where $J_{\alpha}(n)$ is the Jordan's totient function. Thus,
\begin{equation}
J_{\alpha+1}(n)=\sum_{k=1}^{n}\gcd(n,k)J_{\alpha}\left(\gcd(n,k)\right).\label{eq:2.12}
\end{equation}
\end{rem}

\begin{rem}
Since if $f$ is multiplicative, then
\[
\prod_{p\vert n}\left(1-f(p)\right)=\sum_{d\vert n}\mu(d)f(d).
\]
Then for $0<q<1$, $\Re(z)>0$ and any multiplicative function we
have
\begin{equation}
\prod_{n=1}^{\infty}\left(q^{nz};q^{n}\right)_{\infty}^{\mu(n)f(n)/n}=\exp\left(-\sum_{n=1}^{\infty}\frac{\prod_{p\vert n}\left(1-f(p)\right)}{n}\frac{q^{nz}}{1-q^{n}}\right),\label{eq:2.13}
\end{equation}
 
\begin{equation}
\prod_{n=1}^{\infty}\left(\frac{\left(q^{n(z+1)};q^{2n}\right)_{\infty}}{\left(q^{nz};q^{2n}\right)_{\infty}}\right)^{\mu(n)f(n)/n}=\exp\left(\sum_{n=1}^{\infty}\frac{\prod_{p\vert n}\left(1-f(p)\right)}{n}\frac{q^{nz}}{1+q^{n}}\right).\label{eq:2.14}
\end{equation}
\end{rem}

\begin{cor}
If 
\begin{equation}
\sum_{n=1}^{\infty}\frac{|f(n)|}{n^{2}}<\infty,\quad\sum_{n=1}^{\infty}\frac{f(n)}{n}w^{n}=o\left(\frac{1}{1-w}\right),\ w\to1^{-}.\label{eq:2.15-1}
\end{equation}
Then for $\Re(z)>0$ we have
\begin{equation}
\lim_{q\uparrow1}\prod_{n=1}^{\infty}\left(q^{nz};q^{n}\right)_{\infty}^{(1-q)g(n)/n}=\exp\left(-\sum_{n=1}^{\infty}\frac{f(n)}{n^{2}}\right).\label{eq:2.15}
\end{equation}
\end{cor}

\begin{proof}
Since for $n\in\mathbb{N},\ 0<q<1$,
\[
\frac{1-q^{n}}{n(1-q)}=\frac{1+q+\dots+q^{n-1}}{n}
\]
and
\[
q^{(n-1)/2}=\sqrt[n]{q^{n(n-1)/2}}<\frac{1+q+\dots+q^{n-1}}{n}<1,
\]
 then
\[
nq^{(n-1)/2}<\frac{1-q^{n}}{(1-q)}<n,\quad0<q<1,\ n\in\mathbb{N}.
\]
Then for $\Re(z)\ge\frac{1}{2}$ we have
\[
\sum_{n=1}^{\infty}\left|\frac{f(n)}{n}\frac{(1-q)q^{nz}}{1-q^{n}}\right|\le\sqrt{q}\sum_{n=1}^{\infty}\frac{|f(n)|q^{\left(\Re(z)-1/2\right)n}}{n^{2}}<\sum_{n=1}^{\infty}\frac{|f(n)|}{n^{2}}.
\]
Thus for $\Re(z)>0$ we have
\[
\lim_{q\to1^{-}}\sum_{n=1}^{\infty}\frac{f(n)}{n}\frac{(1-q)q^{n(z+1)}}{1-q^{n}}=\sum_{n=1}^{\infty}\frac{f(n)}{n^{2}}
\]
 uniformly in $z$.

By
\begin{align*}
 & \sum_{n=1}^{\infty}\frac{f(n)}{n}\frac{(1-q)q^{nz}}{1-q^{n}}=(1-q)\sum_{n=1}^{\infty}\frac{f(n)}{n}q^{nz}+\sum_{n=1}^{\infty}\frac{f(n)}{n}\frac{(1-q)q^{n(z+1)}}{1-q^{n}}\\
 & =o\left(\frac{1-q}{1-q^{z}}\right)+\sum_{n=1}^{\infty}\frac{f(n)}{n}\frac{(1-q)q^{n(z+1)}}{1-q^{n}}=o(1)+\sum_{n=1}^{\infty}\frac{f(n)}{n}\frac{(1-q)q^{n(z+1)}}{1-q^{n}}
\end{align*}
we have
\[
\lim_{q\to1^{-}}\sum_{n=1}^{\infty}\frac{f(n)}{n}\frac{(1-q)q^{nz}}{1-q^{n}}=\sum_{n=1}^{\infty}\frac{f(n)}{n^{2}}
\]

\[
\prod_{n=1}^{\infty}\left(q^{nz};q^{n}\right)_{\infty}^{g(n)/n}=\exp\left(-\sum_{n=1}^{\infty}\frac{f(n)}{n}\frac{q^{nz}}{1-q^{n}}\right)
\]
\end{proof}
Clearly, (\ref{eq:2.15-1}) is satisfied if 
\begin{equation}
f(n)=\mathcal{O}\left(\frac{n}{\log^{2}(n+1)}\right),\quad n\to\infty.
\end{equation}
Of course, if $f(n)\ge0$, then (\ref{eq:2.15}) is always true regardless
whether (\ref{eq:2.15-1}) is true or not. 

\section{Examples}

For the brevity we only present selected examples for Theorem \ref{thm:2.1}.
They are related to some well-known arithmetic functions studied in
\cite{Apostol,Hardy,McCarthy,Sandor1,Sandor2,Sivaramakrishnan}, and
in all the examples we assume that $\Re(z)>0$ and $0<q<1$ unless
otherwise stated. 

\subsubsection{The M\"{o}bius function $\mu(n)$.}

The M\"{o}bius function $\mu(n)$ satisfies
\[
\sum_{d\mid n}\mu(d)=\begin{cases}
1 & \text{ if }n=1\\
0 & \text{ if }n\neq1
\end{cases},
\]
to get
\begin{equation}
\prod_{n=1}^{\infty}\left(q^{nz};q^{n}\right)_{\infty}^{\mu(n)/n}=\exp\left(-\frac{q^{z}}{1-q}\right).\label{eq:3.1}
\end{equation}
\begin{equation}
\prod_{n=1}^{\infty}\left(\frac{\left(q^{n(z+1)};q^{2n}\right)_{\infty}}{\left(q^{nz};q^{2n}\right)_{\infty}}\right)^{\mu(n)/n}=\exp\left(\frac{q^{z}}{1+q}\right).\label{eq:3.2}
\end{equation}
Then,
\begin{equation}
\lim_{q\uparrow1}\prod_{n=1}^{\infty}\left(q^{nz};q^{n}\right)_{\infty}^{(1-q)\mu(n)/n}=e^{-1}\label{eq:3.1a}
\end{equation}
and
\begin{equation}
\lim_{q\uparrow1}\prod_{n=1}^{\infty}\left(\frac{\left(q^{n(z+1)};q^{2n}\right)_{\infty}}{\left(q^{nz};q^{2n}\right)_{\infty}}\right)^{\mu(n)/n}=e^{1/2}.\label{eq:3.2a}
\end{equation}
Let $\omega(n)$ equal the number of distinct prime factors of $n$,
then, 

\[
\left(-1\right)^{\omega(n)}=\sum_{d\vert n}2^{\omega(d)}\mu(d).
\]
Thus,
\begin{equation}
\prod_{n=1}^{\infty}\left(q^{nz};q^{n}\right)_{\infty}^{2^{\omega(n)}\mu(n)/n}=\exp\left(-\sum_{n=1}^{\infty}\frac{\left(-1\right)^{\omega(n)}}{n}\frac{q^{nz}}{1-q^{n}}\right)\label{eq:3.3}
\end{equation}
and
\begin{equation}
\prod_{n=1}^{\infty}\left(\frac{\left(q^{n(z+1)};q^{2n}\right)_{\infty}}{\left(q^{nz};q^{2n}\right)_{\infty}}\right)^{2^{\omega(n)}\mu(n)/n}=\exp\left(\sum_{n=1}^{\infty}\frac{\left(-1\right)^{\omega(n)}}{n}\frac{q^{nz}}{1+q^{n}}\right).\label{eq:3.4}
\end{equation}
Since $\sum_{n=1}^{\infty}\frac{\left(-1\right)^{\omega(n)}}{n^{2}}$
converges absolutely, then
\begin{equation}
\lim_{q\uparrow1}\prod_{n=1}^{\infty}\left(q^{nz};q^{n}\right)_{\infty}^{2^{\omega(n)}\mu(n)(1-q)/n}=\exp\left(-\sum_{n=1}^{\infty}\frac{\left(-1\right)^{\omega(n)}}{n^{2}}\right).\label{eq:3.3a}
\end{equation}
Let $\theta(n)$ be the number of ordered pairs $\left(a,b\right)$
of positive integers such that 
\[
\gcd(a,b)=1,\quad n=ab.
\]
 then
\[
\theta(n)=2^{\omega(n)}=\sum_{d\vert n}\left|\mu(n)\right|.
\]
 Thus,
\begin{equation}
\prod_{n=1}^{\infty}\left(q^{nz};q^{n}\right)_{\infty}^{\left|\mu(n)\right|/n}=\exp\left(-\sum_{n=1}^{\infty}\frac{2^{\omega(n)}}{n}\frac{q^{nz}}{1-q^{n}}\right)\label{eq:3.5}
\end{equation}
 and
\begin{equation}
\prod_{n=1}^{\infty}\left(\frac{\left(q^{n(z+1)};q^{2n}\right)_{\infty}}{\left(q^{nz};q^{2n}\right)_{\infty}}\right)^{\left|\mu(n)\right|/n}=\exp\left(\sum_{n=1}^{\infty}\frac{2^{\omega(n)}}{n}\frac{q^{nz}}{1+q^{n}}\right).\label{eq:3.6}
\end{equation}
From 
\[
\sum_{n=1}^{\infty}\frac{2^{\omega(n)}}{n^{s}}=\frac{\left(\zeta(s)\right)^{2}}{\zeta(2s)},\quad s>1
\]
we have
\[
\sum_{n=1}^{\infty}\frac{2^{\omega(n)}}{n^{2}}=\frac{5}{2},\quad\sum_{n=1}^{\infty}\frac{2^{\omega(n)}}{n}=\infty.
\]
Then
\begin{equation}
\lim_{q\uparrow1}\prod_{n=1}^{\infty}\left(q^{nz};q^{n}\right)_{\infty}^{\left|\mu(n)\right|(1-q)/n}=e^{-5/2}\label{eq:3.5a}
\end{equation}
and
\begin{equation}
\lim_{q\uparrow1}\prod_{n=1}^{\infty}\left(\frac{\left(q^{n(z+1)};q^{2n}\right)_{\infty}}{\left(q^{nz};q^{2n}\right)_{\infty}}\right)^{\left|\mu(n)\right|/n}=\infty.
\end{equation}

\subsubsection{von Mangoldt function $\Lambda(n)$}

From 
\[
\log n=\sum_{d\vert n}\Lambda(d)
\]
 to get
\begin{equation}
\prod_{n=1}^{\infty}\left(q^{nz};q^{n}\right)_{\infty}^{\Lambda(n)/n}=\exp\left(-\sum_{n=1}^{\infty}\frac{\log(n)}{n}\frac{q^{nz}}{1-q^{n}}\right)\label{eq:3.7}
\end{equation}
and
\begin{equation}
\prod_{n=1}^{\infty}\left(\frac{\left(q^{n(z+1)};q^{2n}\right)_{\infty}}{\left(q^{nz};q^{2n}\right)_{\infty}}\right)^{\Lambda(n)/n}=\exp\left(\sum_{n=1}^{\infty}\frac{\log(n)}{n}\frac{q^{nz}}{1+q^{n}}\right).\label{eq:3.8}
\end{equation}
It is known that
\[
\sum_{n=1}^{\infty}\frac{\log n}{n^{2}}=-\frac{\pi^{2}}{6}\left(\log\left(\frac{2\pi}{A^{12}}\right)+\gamma\right),
\]
where $\gamma=0.577\dots$ is the Euler constant, and $A=1.282\dots$
the Glaisher constant. Then 
\begin{equation}
\lim_{q\uparrow1}\prod_{n=1}^{\infty}\left(q^{nz};q^{n}\right)_{\infty}^{\Lambda(n)(1-q)/n}=\left(\frac{2\pi e^{\gamma}}{A^{12}}\right)^{\pi^{2}/6}\label{eq:3.7a}
\end{equation}
and
\begin{equation}
\lim_{q\uparrow1}\prod_{n=1}^{\infty}\left(\frac{\left(q^{n(z+1)};q^{2n}\right)_{\infty}}{\left(q^{nz};q^{2n}\right)_{\infty}}\right)^{\Lambda(n)/n}=\infty.
\end{equation}
From 
\[
\Lambda(n)=-\sum_{d\vert n}\mu(d)\log d
\]
 to get
\begin{equation}
\prod_{n=1}^{\infty}\left(q^{nz};q^{n}\right)_{\infty}^{\mu(n)\log n/n}=\exp\left(\sum_{n=1}^{\infty}\frac{\Lambda(n)}{n}\frac{q^{nz}}{1-q^{n}}\right)\label{eq:3.9}
\end{equation}
and
\begin{equation}
\prod_{n=1}^{\infty}\left(\frac{\left(q^{n(z+1)};q^{2n}\right)_{\infty}}{\left(q^{nz};q^{2n}\right)_{\infty}}\right)^{\mu(n)\log n/n}=\exp\left(-\sum_{n=1}^{\infty}\frac{\Lambda(n)}{n}\frac{q^{nz}}{1+q^{n}}\right).\label{eq:3.10}
\end{equation}
 By
\[
\sum_{n=1}^{\infty}\frac{\Lambda(n)}{n^{2}}=\log\left(\frac{A^{12}}{2\pi}\right)-\gamma,\quad\sum_{n=1}^{\infty}\frac{\Lambda(n)}{n}=\infty,
\]
then 
\begin{equation}
\lim_{q\uparrow1}\prod_{n=1}^{\infty}\left(q^{nz};q^{n}\right)_{\infty}^{(1-q)\mu(n)\log n/n}=\frac{A^{12}}{2\pi e^{\gamma}}\label{eq:3.9a}
\end{equation}
and
\begin{equation}
\lim_{q\uparrow1}\prod_{n=1}^{\infty}\left(\frac{\left(q^{n(z+1)};q^{2n}\right)_{\infty}}{\left(q^{nz};q^{2n}\right)_{\infty}}\right)^{\mu(n)\log n/n}=0.
\end{equation}

\subsubsection{The Euler's totient function $\varphi(n)$ }

From
\[
\sum_{d\vert n}\varphi(d)=n
\]
 to get
\begin{equation}
\prod_{n=1}^{\infty}\left(q^{nz};q^{n}\right)_{\infty}^{\varphi(n)/n}=\exp\left(-\sum_{n=1}^{\infty}\frac{q^{nz}}{1-q^{n}}\right)\label{eq:3.11}
\end{equation}
 and
\begin{equation}
\prod_{n=1}^{\infty}\left(\frac{\left(q^{n(z+1)};q^{2n}\right)_{\infty}}{\left(q^{nz};q^{2n}\right)_{\infty}}\right)^{\varphi(n)/n}=\exp\left(\sum_{n=1}^{\infty}\frac{q^{nz}}{1+q^{n}}\right).\label{eq:3.12}
\end{equation}
Then,
\[
\lim_{q\uparrow1}\prod_{n=1}^{\infty}\left(q^{nz};q^{n}\right)_{\infty}^{(1-q)\varphi(n)/n}=0
\]
 and
\[
\lim_{q\uparrow1}\prod_{n=1}^{\infty}\left(\frac{\left(q^{n(z+1)};q^{2n}\right)_{\infty}}{\left(q^{nz};q^{2n}\right)_{\infty}}\right)^{\varphi(n)/n}=\infty.
\]
From
\[
\frac{\varphi(n)}{n}=\sum_{d\vert n}\frac{\mu(d)}{d}
\]
 to get
\begin{equation}
\prod_{n=1}^{\infty}\left(q^{nz};q^{n}\right)_{\infty}^{\mu(n)/n^{2}}=\exp\left(-\sum_{n=1}^{\infty}\frac{\varphi(n)}{n^{2}}\frac{q^{nz}}{1-q^{n}}\right)\label{eq:3.13}
\end{equation}
 and
\begin{equation}
\prod_{n=1}^{\infty}\left(\frac{\left(q^{n(z+1)};q^{2n}\right)_{\infty}}{\left(q^{nz};q^{2n}\right)_{\infty}}\right)^{\mu(n)/n^{2}}=\exp\left(\sum_{n=1}^{\infty}\frac{\varphi(n)}{n^{2}}\frac{q^{nz}}{1+q^{n}}\right).\label{eq:3.14}
\end{equation}
 By
\[
\sum_{n=1}^{\infty}\frac{\varphi(n)}{n^{s}}=\frac{\zeta(s-1)}{\zeta(s)},\quad s>2
\]
we get 
\begin{equation}
\lim_{q\uparrow1}\prod_{n=1}^{\infty}\left(q^{nz};q^{n}\right)_{\infty}^{(1-q)\mu(n)/n^{2}}=\exp\left(-\frac{\zeta(2)}{\zeta(3)}\right)\label{eq:3.13-1}
\end{equation}
and
\begin{equation}
\lim_{q\uparrow1}\prod_{n=1}^{\infty}\left(\frac{\left(q^{n(z+1)};q^{2n}\right)_{\infty}}{\left(q^{nz};q^{2n}\right)_{\infty}}\right)^{\mu(n)/n^{2}}=\infty.
\end{equation}
From
\[
\frac{n}{\varphi(n)}=\sum_{d\vert n}\frac{\mu^{2}(d)}{\varphi(d)}
\]
to get
\begin{equation}
\prod_{n=1}^{\infty}\left(q^{nz};q^{n}\right)_{\infty}^{\mu^{2}(n)/(n\varphi(n))}=\exp\left(-\sum_{n=1}^{\infty}\frac{1}{\varphi(n)}\frac{q^{nz}}{1-q^{n}}\right)\label{eq:3.15}
\end{equation}
and
\begin{equation}
\prod_{n=1}^{\infty}\left(\frac{\left(q^{n(z+1)};q^{2n}\right)_{\infty}}{\left(q^{nz};q^{2n}\right)_{\infty}}\right)^{\mu^{2}(n)/(n\varphi(n))}=\exp\left(\sum_{n=1}^{\infty}\frac{1}{\varphi(n)}\frac{q^{nz}}{1+q^{n}}\right).\label{eq:3.16}
\end{equation}
 Since \cite{Rosser}
\[
{\displaystyle \varphi(n)>\frac{n}{e^{\gamma}\;\log\log n+\frac{3}{\log\log n}}\quad n>2,}
\]
 then 
\[
\sum_{n=1}^{\infty}\frac{1}{n\varphi(n)}
\]
 converges. On the other hand, since $\varphi(n)\le n$, then,
\[
\sum_{n=1}^{\infty}\frac{1}{\varphi(n)}=\infty.
\]
Thus,
\begin{equation}
\lim_{q\uparrow1}\prod_{n=1}^{\infty}\left(q^{nz};q^{n}\right)_{\infty}^{(1-q)\mu^{2}(n)/(n\varphi(n))}=\exp\left(-\sum_{n=1}^{\infty}\frac{1}{n\varphi(n)}\right)\label{eq:3.15a}
\end{equation}
 and for $z>0$ we have
\begin{equation}
\lim_{q\uparrow1}\prod_{n=1}^{\infty}\left(\frac{\left(q^{n(z+1)};q^{2n}\right)_{\infty}}{\left(q^{nz};q^{2n}\right)_{\infty}}\right)^{\mu^{2}(n)/(n\varphi(n))}=\infty.
\end{equation}

\subsubsection{Jordan's totient function $J_{k}(n)$}

For each $k\in\mathbb{N}$, from 
\[
n^{k}=\sum_{d\vert n}J_{k}(d)
\]
 to get 
\begin{equation}
\prod_{n=1}^{\infty}\left(q^{nz};q^{n}\right)_{\infty}^{J_{k}(n)/n}=\exp\left(-\sum_{n=1}^{\infty}\frac{n^{k-1}q^{nz}}{1-q^{n}}\right)\label{eq:3.17}
\end{equation}
 and
\begin{equation}
\prod_{n=1}^{\infty}\left(\frac{\left(q^{n(z+1)};q^{2n}\right)_{\infty}}{\left(q^{nz};q^{2n}\right)_{\infty}}\right)^{J_{k}(n)/n}=\exp\left(\sum_{n=1}^{\infty}\frac{n^{k-1}q^{nz}}{1+q^{n}}\right).\label{eq:3.18}
\end{equation}
By (\ref{eq:3.17}) and (\ref{eq:3.18}) we get
\begin{equation}
\lim_{q\uparrow1}\prod_{n=1}^{\infty}\left(q^{nz};q^{n}\right)_{\infty}^{(1-q)J_{k}(n)/n}=\exp\left(-\zeta(2-k)\right),\quad\Re(k)<1\label{eq:3.17a}
\end{equation}
 and
\begin{equation}
\lim_{q\uparrow1}\prod_{n=1}^{\infty}\left(\frac{\left(q^{n(z+1)};q^{2n}\right)_{\infty}}{\left(q^{nz};q^{2n}\right)_{\infty}}\right)^{J_{k}(n)/n}=e^{\zeta(1-k)/2},\quad\Re(k)<0.\label{eq:3.18a}
\end{equation}
From 
\[
\frac{J_{k}(n)}{n^{k}}=\sum_{d\vert n}\frac{\mu(d)}{d^{k}}
\]
 to get
\begin{equation}
\prod_{n=1}^{\infty}\left(q^{nz};q^{n}\right)_{\infty}^{\mu(n)/n^{k+1}}=\exp\left(-\sum_{n=1}^{\infty}\frac{J_{k}(n)}{n^{k+1}}\frac{q^{nz}}{1-q^{n}}\right)\label{eq:3.19}
\end{equation}
 and
\begin{equation}
\prod_{n=1}^{\infty}\left(\frac{\left(q^{n(z+1)};q^{2n}\right)_{\infty}}{\left(q^{nz};q^{2n}\right)_{\infty}}\right)^{\mu(n)/n^{k+1}}=\exp\left(\sum_{n=1}^{\infty}\frac{J_{k}(n)}{n^{k+1}}\frac{q^{nz}}{1+q^{n}}\right).\label{eq:3.20}
\end{equation}
From
\[
\sum_{n=1}^{\infty}\frac{J_{k}(n)}{n^{s}}=\frac{\zeta(s-k)}{\zeta(s)},\quad\Re(k)<0,\ \Re(s)>1,
\]
 we get
\[
\sum_{n=1}^{\infty}\frac{J_{k}(n)}{n^{k+2}}=\frac{\zeta(2)}{\zeta(k+2)},\quad\Re(k)>-1
\]
converges absolutely. Then for $\Re(k)>-1$,
\begin{equation}
\lim_{q\uparrow1}\prod_{n=1}^{\infty}\left(q^{nz};q^{n}\right)_{\infty}^{(1-q)\mu(n)/n^{k+1}}=\exp\left(-\zeta(2)/\zeta(k+2)\right)\label{eq:3.19a}
\end{equation}
and $k>1,\ z>0$,
\[
\lim_{q\uparrow1}\prod_{n=1}^{\infty}\left(\frac{\left(q^{n(z+1)};q^{2n}\right)_{\infty}}{\left(q^{nz};q^{2n}\right)_{\infty}}\right)^{\mu(n)/n^{k+1}}=\infty.
\]
For $k\in\mathbb{N},$from \cite{McCarthy}
\[
\frac{J_{2k}(n)}{J_{k}(n)}=n^{k}\sum_{d\vert n}\frac{\left|\mu(d)\right|}{d^{k}}
\]
to get
\begin{equation}
\prod_{n=1}^{\infty}\left(q^{nz};q^{n}\right)_{\infty}^{|\mu(d)|/n^{k+1}}=\exp\left(-\sum_{n=1}^{\infty}\frac{J_{2k}(n)}{n^{k+1}J_{k}(n)}\frac{q^{nz}}{1-q^{n}}\right)\label{eq:3.21}
\end{equation}
 and
\begin{equation}
\prod_{n=1}^{\infty}\left(\frac{\left(q^{n(z+1)};q^{2n}\right)_{\infty}}{\left(q^{nz};q^{2n}\right)_{\infty}}\right)^{|\mu(d)|/n^{k+1}}=\exp\left(\sum_{n=1}^{\infty}\frac{J_{2k}(n)}{n^{k+1}J_{k}(n)}\frac{q^{nz}}{1+q^{n}}\right).\label{eq:3.22}
\end{equation}
Since
\[
\sum_{n=1}^{\infty}\frac{J_{2k}(n)}{n^{s}J_{k}(n)}=\frac{\zeta(s)\zeta(s-k)}{\zeta(2s)},\quad\Re(s-k)>1,\ \Re(s)>1,
\]
then
\[
\sum_{n=1}^{\infty}\frac{J_{2k}(n)}{n^{k+2}J_{k}(n)}=\frac{\zeta(2)\zeta(k+2)}{\zeta(2k+4)},\quad\Re(k)>-1
\]
converges absolutely, then for $\Re(k)>-1$ we have
\begin{equation}
\lim_{q\uparrow1}\prod_{n=1}^{\infty}\left(q^{nz};q^{n}\right)_{\infty}^{(1-q)|\mu(d)|/n^{k+1}}=\exp\left(-\frac{\zeta(2)\zeta(k+2)}{\zeta(2k+4)}\right)\label{eq:3.21a}
\end{equation}
and for $k>1,z>0$ we have
\[
\lim_{q\uparrow1}\prod_{n=1}^{\infty}\left(\frac{\left(q^{n(z+1)};q^{2n}\right)_{\infty}}{\left(q^{nz};q^{2n}\right)_{\infty}}\right)^{|\mu(d)|/n^{k+1}}=\infty.
\]

\subsubsection{The divisor functions }

From 
\[
d(n^{2})=\sum_{d\vert n}2^{\omega(d)}
\]
 to get
\begin{equation}
\prod_{n=1}^{\infty}\left(q^{nz};q^{n}\right)_{\infty}^{2^{\omega(n)}/n}=\exp\left(-\sum_{n=1}^{\infty}\frac{d(n^{2})}{n}\frac{q^{nz}}{1-q^{n}}\right)\label{eq:3.23}
\end{equation}
 and
\begin{equation}
\prod_{n=1}^{\infty}\left(\frac{\left(q^{n(z+1)};q^{2n}\right)_{\infty}}{\left(q^{nz};q^{2n}\right)_{\infty}}\right)^{2^{\omega(n)}/n}=\exp\left(\sum_{n=1}^{\infty}\frac{d(n^{2})}{n}\frac{q^{nz}}{1+q^{n}}\right).\label{eq:3.24}
\end{equation}
By
\[
\sum_{n=1}^{\infty}\frac{d(n^{2})}{n^{s}}=\frac{\zeta^{3}(s)}{\zeta(2s)},\quad\Re(s)>1
\]
 to get
\[
\sum_{n=1}^{\infty}\frac{d(n^{2})}{n^{2}}=\frac{5\pi^{2}}{12},\quad\sum_{n=1}^{\infty}\frac{d(n^{2})}{n}=\infty,
\]
then
\begin{equation}
\lim_{q\uparrow1}\prod_{n=1}^{\infty}\left(q^{nz};q^{n}\right)_{\infty}^{(1-q)2^{\omega(n)}/n}=\exp\left(-\frac{5\pi^{2}}{12}\right)\label{eq:3.23a}
\end{equation}
and for $z>0$ we have
\begin{equation}
\lim_{q\uparrow1}\prod_{n=1}^{\infty}\left(\frac{\left(q^{n(z+1)};q^{2n}\right)_{\infty}}{\left(q^{nz};q^{2n}\right)_{\infty}}\right)^{2^{\omega(n)}/n}=\infty.
\end{equation}
From 
\[
\sum_{\delta\vert n}d(\delta^{2})=d^{2}(n)
\]
to get
\begin{equation}
\prod_{n=1}^{\infty}\left(q^{nz};q^{n}\right)_{\infty}^{d(n^{2})/n}=\exp\left(-\sum_{n=1}^{\infty}\frac{d^{2}(n)}{n}\frac{q^{nz}}{1-q^{n}}\right)\label{eq:3.25}
\end{equation}
 and
\begin{equation}
\prod_{n=1}^{\infty}\left(\frac{\left(q^{n(z+1)};q^{2n}\right)_{\infty}}{\left(q^{nz};q^{2n}\right)_{\infty}}\right)^{d(n^{2})/n}=\exp\left(\sum_{n=1}^{\infty}\frac{d^{2}(n)}{n}\frac{q^{nz}}{1+q^{n}}\right).\label{eq:3.26}
\end{equation}
By
\[
\sum_{n=1}^{\infty}\frac{d^{2}(n)}{n^{s}}=\frac{\zeta^{4}(s)}{\zeta(2s)},\quad\Re(s)>1
\]
 to get
\[
\sum_{n=1}^{\infty}\frac{d^{2}(n)}{n^{2}}=\frac{5\pi^{4}}{72},\quad\sum_{n=1}^{\infty}\frac{d^{2}(n)}{n}=\infty.
\]
Then,
\begin{equation}
\lim_{q\uparrow1}\prod_{n=1}^{\infty}\left(q^{nz};q^{n}\right)_{\infty}^{(1-q)d(n^{2})/n}=\exp\left(-\frac{5\pi^{4}}{72}\right)\label{eq:3.25a}
\end{equation}
and
\begin{equation}
\lim_{q\uparrow1}\prod_{n=1}^{\infty}\left(\frac{\left(q^{n(z+1)};q^{2n}\right)_{\infty}}{\left(q^{nz};q^{2n}\right)_{\infty}}\right)^{d(n^{2})/n}=\infty.
\end{equation}
From 
\[
{\displaystyle \sigma_{s}(n)=\sum_{d\mid n}d^{s},\quad s\in\mathbb{C}}
\]
to get
\begin{equation}
\prod_{n=1}^{\infty}\left(q^{nz};q^{n}\right)_{\infty}^{n^{s-1}}=\exp\left(-\sum_{n=1}^{\infty}\frac{\sigma_{s}(n)}{n}\frac{q^{nz}}{1-q^{n}}\right)\label{eq:3.27}
\end{equation}
 and
\begin{equation}
\prod_{n=1}^{\infty}\left(\frac{\left(q^{n(z+1)};q^{2n}\right)_{\infty}}{\left(q^{nz};q^{2n}\right)_{\infty}}\right)^{n^{s-1}}=\exp\left(\sum_{n=1}^{\infty}\frac{\sigma_{s}(n)}{n}\frac{q^{nz}}{1+q^{n}}\right).\label{eq:3.28}
\end{equation}
By
\[
\sum_{n=1}^{\infty}\frac{\sigma_{a}(n)}{n^{s}}=\zeta(s)\zeta(s-a),\quad\Re(s)>1,\ \Re(s-a)>1
\]
we have
\[
\sum_{n=1}^{\infty}\frac{\sigma_{s}(n)}{n^{2}}=\zeta(2)\zeta(2-s),\quad\Re(s)<1.
\]
Then,
\begin{equation}
\lim_{q\uparrow1}\prod_{n=1}^{\infty}\left(q^{nz};q^{n}\right)_{\infty}^{(1-q)n^{s-1}}=\exp\left(-\zeta(2)\zeta(2-s)\right),\quad\Re(s)<1\label{eq:3.27-1}
\end{equation}
 and for $s<0,\ z>0$ we have
\begin{equation}
\lim_{q\uparrow1}\prod_{n=1}^{\infty}\left(\frac{\left(q^{n(z+1)};q^{2n}\right)_{\infty}}{\left(q^{nz};q^{2n}\right)_{\infty}}\right)^{n^{s-1}}=\infty.
\end{equation}

\subsubsection{Liouville function }

From 
\[
\sum_{d\mid n}\lambda(d)=\begin{cases}
1 & \text{ if \ensuremath{n} is a square }\\
0 & \text{ if \ensuremath{n} is not a square }
\end{cases}
\]
 to get
\begin{equation}
\prod_{n=1}^{\infty}\left(q^{nz};q^{n}\right)_{\infty}^{\lambda(n)/n}=\exp\left(-\sum_{n=1}^{\infty}\frac{q^{n^{2}z}}{n^{2}\left(1-q^{n^{2}}\right)}\right)\label{eq:3.29}
\end{equation}
 and
\begin{equation}
\prod_{n=1}^{\infty}\left(\frac{\left(q^{n(z+1)};q^{2n}\right)_{\infty}}{\left(q^{nz};q^{2n}\right)_{\infty}}\right)^{\lambda(n)/n}=\exp\left(\sum_{n=1}^{\infty}\frac{q^{n^{2}z}}{n^{2}\left(1+q^{n^{2}}\right)}\right).\label{eq:3.30}
\end{equation}
 Then,
\begin{equation}
\lim_{q\uparrow1}\prod_{n=1}^{\infty}\left(q^{nz};q^{n}\right)_{\infty}^{(1-q)\lambda(n)/n}=\exp\left(\frac{\pi^{4}}{90}\right)\label{eq:3.29-1}
\end{equation}
and
\begin{equation}
\lim_{q\uparrow1}\prod_{n=1}^{\infty}\left(\frac{\left(q^{n(z+1)};q^{2n}\right)_{\infty}}{\left(q^{nz};q^{2n}\right)_{\infty}}\right)^{\lambda(n)/n}=\exp\left(\frac{\pi^{2}}{12}\right).\label{eq:3.30-1}
\end{equation}

\subsubsection{Ramanujan's sum }

For each $n,q\in\mathbb{N}$, the Ramanujan's sum $c_{q}(n)$ defined
by
\[
c_{q}(n)=\sum_{{a=1\atop \gcd(a,q)=1}}^{q}e^{2\pi ian/q}.
\]
It satisfies 
\[
\sum_{d\mid u}c_{d}(v)=\begin{cases}
0 & \;\mbox{ if }u\nmid v\\
u & \;\mbox{ if }u\mid v
\end{cases},\quad u,v\in\mathbb{N}.
\]
Then for $v\in\mathbb{N}$ we have
\begin{equation}
\prod_{n=1}^{\infty}\left(q^{nz};q^{n}\right)_{\infty}^{c_{n}(v)/n}=\exp\left(-\sum_{n\vert v}\frac{q^{nz}}{1-q^{n}}\right)\label{eq:3.31}
\end{equation}
 and
\begin{equation}
\prod_{n=1}^{\infty}\left(\frac{\left(q^{n(z+1)};q^{2n}\right)_{\infty}}{\left(q^{nz};q^{2n}\right)_{\infty}}\right)^{c_{n}(v)/n}=\exp\left(\sum_{n\vert v}\frac{q^{nz}}{1+q^{n}}\right).\label{eq:3.32}
\end{equation}
 Thus,
\begin{equation}
\lim_{q\uparrow1}\prod_{n=1}^{\infty}\left(q^{nz};q^{n}\right)_{\infty}^{(1-q)c_{n}(v)/n}=\exp\left(-\sigma_{-1}(v)\right)\label{eq:3.31-1}
\end{equation}
 and
\begin{equation}
\lim_{q\uparrow1}\prod_{n=1}^{\infty}\left(\frac{\left(q^{n(z+1)};q^{2n}\right)_{\infty}}{\left(q^{nz};q^{2n}\right)_{\infty}}\right)^{c_{n}(v)/n}=\exp\left(\frac{d(v)}{2}\right).\label{eq:3.32-1}
\end{equation}

\subsubsection{Sum of $2,4,8$ squares}

For $n,k\in\mathbb{N}$ the arithmetic function $r_{k}(n)$ is defined
as the number of ways $n$ can be represented as the sum $k$ squares.
Let 
\begin{align*}
\chi_{1}(n) & =\begin{cases}
0, & n\equiv0\mod4,\\
1 & n\equiv1\mod4,\\
0 & n\equiv2\mod4,\\
-1 & n\equiv3\mod4.
\end{cases}
\end{align*}
 Then
\[
\frac{r_{2}(n)}{4}=\sum_{d\mid n}\chi_{1}(d).
\]
 Hence,
\begin{equation}
\prod_{k=1}^{\infty}\frac{\left(q^{(4k-1)z};q^{4k-1}\right)_{\infty}^{4/(4k-1)}}{\left(q^{(4k-3)z};q^{4k-3}\right)_{\infty}^{4/(4k-3)}}=\exp\left(\sum_{n=1}^{\infty}\frac{r_{2}(n)}{n}\frac{q^{nz}}{1-q^{n}}\right)\label{eq:3.33}
\end{equation}
 and
\begin{equation}
\prod_{k=1}^{\infty}\frac{\left(q^{(4k-1)z};q^{8k-2}\right)_{\infty}^{\frac{4}{4k-1}}\left(q^{(4k-3)(z+1)};q^{8k-6}\right)_{\infty}^{\frac{4}{4k-3}}}{\left(q^{(4k-1)(z+1)};q^{8k-2}\right)_{\infty}^{\frac{4}{4k-1}}\left(q^{(4k-3)z};q^{8k-6}\right)_{\infty}^{\frac{4}{4k-3}}}=\exp\left(\sum_{n=1}^{\infty}\frac{r_{2}(n)q^{nz}}{n(1+q^{n})}\right).\label{eq:3.34}
\end{equation}
Since
\[
\sum_{n=1}^{\infty}\frac{r_{2}(n)}{n^{s}}=4\zeta(s)\beta(s),\quad\Re(s)>1,
\]
where
\[
\beta(s)=\sum_{n=0}^{\infty}\frac{(-1)^{n}}{(2n+1)^{s}},\quad\Re(s)>0
\]
and
\[
\beta(1)=\frac{\pi}{4},\quad\beta(2)=G,
\]
 where $G=0.915\dots$ is the Catalan constant. Then,
\[
\sum_{n=1}^{\infty}\frac{r_{2}(n)}{n^{2}}=\frac{2}{3}\pi^{2}G,\quad\sum_{n=1}^{\infty}\frac{r_{2}(n)}{n}=\infty.
\]
 Hence,
\begin{equation}
\lim_{q\uparrow1}\prod_{k=1}^{\infty}\frac{\left(q^{(4k-1)z};q^{4k-1}\right)_{\infty}^{4(1-q)/(4k-1)}}{\left(q^{(4k-3)z};q^{4k-3}\right)_{\infty}^{4(1-q)/(4k-3)}}=\exp\left(\frac{2}{3}\pi^{2}G\right)\label{eq:3.33-1}
\end{equation}
and
\begin{equation}
\lim_{q\uparrow1}\prod_{k=1}^{\infty}\frac{\left(q^{(4k-1)z};q^{8k-2}\right)_{\infty}^{\frac{4}{4k-1}}\left(q^{(4k-3)(z+1)};q^{8k-6}\right)_{\infty}^{\frac{4}{4k-3}}}{\left(q^{(4k-1)(z+1)};q^{8k-2}\right)_{\infty}^{\frac{4}{4k-1}}\left(q^{(4k-3)z};q^{8k-6}\right)_{\infty}^{\frac{4}{4k-3}}}=\infty.
\end{equation}
From 

\[
\frac{r_{4}(n)}{8}=\sum_{\stackrel{d\mid n}{4\,\nmid\,d}}d
\]
 to get 
\begin{align}
 & \exp\left(-\sum_{n=1}^{\infty}\frac{r_{4}(n)}{8n}\frac{q^{nz}}{1-q^{n}}\right)=\prod_{k=1}^{\infty}\left(q^{(4k-3)z};q^{4k-3}\right)_{\infty}\label{eq:3.35}\\
 & \times\left(q^{(4k-2)z};q^{4k-2}\right)_{\infty}\left(q^{(4k-1)z};q^{4k-1}\right)_{\infty}\nonumber 
\end{align}
and
\begin{align}
 & \exp\left(\sum_{n=1}^{\infty}\frac{r_{4}(n)}{8n}\frac{q^{nz}}{1+q^{n}}\right)=\prod_{k=1}^{\infty}\frac{\left(q^{(4k-1)(z+1)};q^{8k-2}\right)_{\infty}}{\left(q^{(4k-1)z};q^{8k-2}\right)_{\infty}}\label{eq:3.36}\\
 & \times\frac{\left(q^{(4k-2)(z+1)};q^{8k-4}\right)_{\infty}}{\left(q^{(4k-2)z};q^{8k-4}\right)_{\infty}}\frac{\left(q^{(4k-3)(z+1)};q^{8k-6}\right)_{\infty}}{\left(q^{(4k-3)z};q^{8k-6}\right)_{\infty}}.\nonumber 
\end{align}
From 
\[
r_{4}(n)=8(2+(-1)^{n})\sum_{\stackrel{d\mid n}{2\,\nmid\,d}}d
\]
 to get
\begin{equation}
\prod_{k=1}^{\infty}\left(q^{(2k-1)z};q^{2k-1}\right)_{\infty}=\exp\left(-\sum_{n=1}^{\infty}\frac{r_{4}(n)}{8n(2+(-1)^{n})}\frac{q^{nz}}{1-q^{n}}\right)\label{eq:3.37}
\end{equation}
and
\begin{equation}
\prod_{k=1}^{\infty}\frac{\left(q^{(2k-1)(z+1)};q^{4k-2}\right)_{\infty}}{\left(q^{(2k-1)z};q^{4k-2}\right)_{\infty}}=\exp\left(\sum_{n=1}^{\infty}\frac{r_{4}(n)}{8n(2+(-1)^{n})}\frac{q^{nz}}{1+q^{n}}\right).\label{eq:3.38}
\end{equation}
From
\[
\left(-1\right)^{n}r_{8}(n)=\sum_{d\vert n}\left(-1\right)^{d}d^{3},
\]
 we get
\begin{equation}
\prod_{n=1}^{\infty}\frac{\left(q^{(2n-1)z};q^{2n-1}\right)_{\infty}^{(2n-1)^{2}}}{\left(q^{2nz};q^{2n}\right)_{\infty}^{4n^{2}}}=\exp\left(-\sum_{n=1}^{\infty}\frac{r_{8}(n)}{n}\frac{(-1)^{n}q^{nz}}{1-q^{n}}\right)\label{eq:3.39}
\end{equation}
 and
\begin{equation}
\prod_{n=1}^{\infty}\frac{\left(q^{2n(z+1)};q^{4n}\right)_{\infty}^{4n^{2}}\left(q^{(2n-1)z};q^{4n-2}\right)_{\infty}^{(2n-1)^{2}}}{\left(q^{2nz};q^{4n}\right)_{\infty}^{4n^{2}}\left(q^{(2n-1)(z+1)};q^{4n-2}\right)_{\infty}^{(2n-1)^{2}}}=\exp\left(\sum_{n=1}^{\infty}\frac{r_{8}(n)}{n}\frac{(-1)^{n}q^{nz}}{1+q^{n}}\right).\label{eq:3.40}
\end{equation}

\subsubsection{The core function}

The core function $\gamma(n)$ is defined by \cite{McCarthy}
\[
\gamma(n)=\begin{cases}
1, & n=1\\
p_{1}\cdots p_{t}, & n=p_{1}^{\alpha_{1}}\cdots p_{t}^{\alpha_{t}}
\end{cases}.
\]
 then $\gamma(n)$ is multiplicative and 
\[
\gamma(n)=\sum_{d\vert n}\varphi(d)\left|\mu(d)\right|.
\]
 Then, 
\begin{equation}
\prod_{n=1}^{\infty}\left(q^{nz};q^{n}\right)_{\infty}^{\varphi(n)\left|\mu(n)\right|/n}=\exp\left(-\sum_{n=1}^{\infty}\frac{\gamma(n)}{n}\frac{q^{nz}}{1-q^{n}}\right)\label{eq:3.41}
\end{equation}
and
\begin{equation}
\prod_{n=1}^{\infty}\left(\frac{\left(q^{n(z+1)};q^{2n}\right)_{\infty}}{\left(q^{nz};q^{2n}\right)_{\infty}}\right)^{\varphi(n)\left|\mu(n)\right|/n}=\exp\left(\sum_{n=1}^{\infty}\frac{\gamma(n)}{n}\frac{q^{nz}}{1+q^{n}}\right).\label{eq:3.42}
\end{equation}

\subsubsection{M\"{o}bius function $\mu_{k}(n)$ }

For $k\in\mathbb{N}$, the generalized M\"{o}bius function $\mu_{k}(n)$
is defined by 
\[
\mu_{k}(n)=\begin{cases}
0, & \text{if \ensuremath{n} contains a squared factor \ensuremath{>1}}\\
\exp\left(\frac{\pi i\omega(n)}{k}\right), & \text{otherwise}
\end{cases}.
\]
Then,
\[
\sum_{d\vert n}\mu_{k}(d)=\begin{cases}
1, & n=1\\
\left(1+\exp\left(\frac{\pi i}{k}\right)\right)^{\omega(n)}, & n>1
\end{cases}.
\]
Thus,
\begin{equation}
\prod_{n=1}^{\infty}\left(q^{nz};q^{n}\right)_{\infty}^{\mu_{k}(n)/n}=\exp\left(-\frac{q^{z}}{1-q}-\sum_{n=2}^{\infty}\frac{\left(1+e^{\pi i/k}\right)^{\omega(n)}q^{nz}}{n\left(1-q^{n}\right)}\right)\label{eq:3.43}
\end{equation}
 and
\begin{equation}
\prod_{n=1}^{\infty}\left(\frac{\left(q^{n(z+1)};q^{2n}\right)_{\infty}}{\left(q^{nz};q^{2n}\right)_{\infty}}\right)^{\mu_{k}(n)/n}=\exp\left(\frac{q^{z}}{1+q}+\sum_{n=2}^{\infty}\frac{\left(1+e^{\pi i/k}\right)^{\omega(n)}q^{nz}}{n\left(1+q^{n}\right)}\right).\label{eq:3.44}
\end{equation}
 Since 
\[
\sum_{n=2}^{\infty}\left|\frac{\left(1+e^{\pi i/k}\right)^{\omega(n)}}{n^{2}}\right|\le\sum_{n=2}^{\infty}\frac{2^{\omega(n)}}{n^{2}}<\frac{5}{2},
\]
then 
\begin{equation}
\lim_{q\uparrow1}\prod_{n=1}^{\infty}\left(q^{nz};q^{n}\right)_{\infty}^{(1-q)\mu_{k}(n)/n}=\exp\left(-1-\sum_{n=2}^{\infty}\frac{\left(1+e^{\pi i/k}\right)^{\omega(n)}}{n^{2}}\right).\label{eq:3.43a}
\end{equation}

\end{document}